\documentclass[12pt]{article}
\usepackage{pgf,tikz}
\usepackage{calc,graphicx, complexity}
\usepackage{graphics}
\everymath{\displaystyle}
\usepackage{graphicx}
\usepackage{color}
\usepackage{amssymb}
\usepackage{amsmath}
\newtheorem{prethm}{{\bf Theorem}}

\newenvironment{thm}{\begin{prethm}{\hspace{-0.5
				em}{\bf .}}}{\end{prethm}}
\newtheorem{prelemma}{{\bf Lemma}}

\newenvironment{lemma}{\begin{prelemma}{\hspace{-0.5
				em}{\bf .}}}{\end{prelemma}}
\newtheorem{preex}{{\bf Example}}

\newtheorem{preprop}{{\bf Proposition}}

\newenvironment{prop}{\begin{preprop}{\hspace{-0.5em}{\bf .}}}{\end{preprop}}
\newtheorem{precor}{{\bf Corollary}}

\newenvironment{cor}{\begin{precor}{\hspace{-0.5
				em}{\bf .}}}{\end{precor}}
\newtheorem{preremark}{{\bf Remark}}

\newtheorem{preprob}{{\bf Problem}}

\newtheorem{predefin}{{\bf Definition}}

\newtheorem{preconj}{{\bf Conjecture}}

\newtheorem{preprobb}{{\bf Problem}}

\newtheorem{prelem}{{\bf Theorem}}

\newtheorem{precla}{{\bf Claim}}

\newenvironment{proof}{{\bf Proof.}\rm }{\hfill{$\Box$}}

\newtheorem{presolution}{{\bf Solution.}}

\def\newpic#1{}
\def\qed{\ifhmode\unskip\nobreak\fi\quad\ifmmode\Box\else$\Box$\fi}

\title{\vspace{0cm}\Large\bf\noindent On ${\rm b}^{\ast}$-Coloring and $z$-Coloring of graphs\\
 with high girth}
\author{\large\bf Zahra Ahmadidahr ~~ Manouchehr Zaker\footnote{mzaker@iasbs.ac.ir}
\vspace{5mm}\\
Department of Mathematics,\\
Institute for Advanced Studies in Basic Sciences,\\
Zanjan 45137-66731, Iran\\
}

\date{}

\begin{document}
\maketitle
\begin{abstract}
\noindent In a proper vertex coloring $c$ of a graph $G$, a vertex $u$ is called a b-vertex if $u$ is adjacent to a vertex in every other color class. A ${\rm b}^{\ast}$-coloring is a proper coloring in which a b-vertex is adjacent to a b-vertex in every other color class. A Grundy coloring is a proper coloring obtained by the First-Fit (greedy) coloring procedure. A $z$-coloring of $G$ is a ${\rm b}^{\ast}$-coloring that is also a Grundy coloring. The ${\rm b}^{\ast}$-chromatic number (resp., $z$-chromatic number), denoted by ${\rm b}^{\ast}(G)$ (resp., $z(G)$), is the maximum number of colors used in a ${\rm b}^{\ast}$-coloring (resp., $z$-coloring) of $G$. Every graph admits a ${\rm b}^{\ast}$-coloring and a $z$-coloring that can be found using a polynomial-time coloring heuristic. Let ${\rm m}^{\ast}(G)$ be the largest integer $k$ such that a vertex of degree at least $k$ in $G$ has $k$ neighbors of degree at least $k$. We employ list-coloring techniques to prove that if $G$ has a girth of at least $7$, then ${\rm b}^{\ast}(G) = {\rm m}^{\ast}(G)+ 1$. A similar result is obtained for graphs of girth at least $6$ when ${\rm m}^{\ast}=3$. Finally, we obtain some results for the $z$-chromatic number. We prove that if the girth is at least $2m^{\ast}(G)+4$ and $G$ contains a specific tree as an ordinary subgraph, then $z(G)= m^{\ast}(G)+1$.
\end{abstract}



\section{Introduction}

\noindent All graphs in this paper are undirected without any loops and multiple edges. In a graph $G$, $d_G(v)$ denotes the degree of a vertex $v$ in $G$. We denote the maximum degree of a graph $G$ by $\Delta(G)$. Complete graph on $n$ vertices is denoted by $K_n$. For a subset $S$ of vertices in $G$, by $G[S]$ we mean the subgraph of $G$ induced on the elements of $S$. By a block in $G$ we mean a maximal subgraph which is either isomorphic to $K_2$ or is $2$-connected. For a vertex $v$ of $G$, define $G-v=G[V(G)\setminus \{v\}]$. A subset of vertices $S$ in a graph $G$ is called independent if the vertices of $S$ do not induce any edge. The girth of a graph $G$ is the length of the shortest cycle in $G$. The radius of $G$ is the minimum eccentricity of any vertex in $G$, where eccentricity is the maximum distance from a vertex to any other vertex. A proper vertex coloring of a graph $G$ is an assignment of colors $c:V(G)\rightarrow \mathbb{N}$ such that no two adjacent vertices receive same colors. For each vertex $v$, the color of $v$ is denoted by $c(v)$. By a color class we mean a subset of vertices having a same color. The chromatic number $\chi(G)$, is the smallest number of colors used in a proper vertex coloring of $G$. We refer to \cite{BM} for the concepts not defined here. In a proper vertex coloring $c$ of $G$, a vertex $u$ is called b-vertex if $u$ has a neighbor of color $j$ for each color $j\not= c(u)$. A proper coloring $c$ is b-coloring if each color class contains a b-vertex. The maximum number of colors in a b-coloring of $G$ is called b-chromatic number and denoted by ${\rm b}(G)$. This quantity is also denoted by $\chi_b(G)$ and $\varphi(G)$. Clearly, ${\rm b}(G)\leq \Delta(G)+1$. The literature is full of papers concerning the b-coloring of graphs e.g. the survey paper \cite{JP} and \cite{BSSV,CLS,DFPRZ,KTV,MS,Z3}. Also to determine ${\rm b}(G)$ is $\NP$-complete for complement of bipartite graphs \cite{BSSV}, for bipartite graphs \cite{KTV}. By a Grundy-coloring (First-Fit coloring) of $G$ we mean a coloring using say $k$ colors such that for each $i,j\in \{1, 2, \ldots, k\}$ with $i<j$, each vertex of color $j$ in $G$ has a neighbor of color $i$. The maximum $k$ satisfying this property is called the Grundy number (or First-Fit chromatic number) of $G$ and denoted by $\Gamma(G)$ also by $\chi_{_{\sf FF}}(G)$. Many researches are devoted to simultaneous or comparative studies of the Grundy and b-chromatic numbers e.g. \cite{SH,Z0}.

\noindent In a proper vertex coloring $c$ of $G$, a vertex is called nice vertex (\cite{Z1}) if for each $j\not= c(u)$, $u$ has a neighbor which is b-vertex of color $j$. A coloring $c$ is called b$^{\ast}$-coloring if there exists a nice vertex in $(G,c)$ \cite{Z2}. It is called ${\rm b}^{\ast}$-coloring because the subgraph induced on b-vertices contains a star graph $K_{1,k-1}$ as subgraph, where the nice vertex is at the center. The ${\rm b}^{\ast}$-chromatic number ${\rm b}^{\ast}(G)$ of $G$ is the maximum number of colors used in a ${\rm b}^{\ast}$-coloring of $G$. It was proved in \cite{Z1,Z2} that there exists an $\mathcal{O}(nm)$ algorithm such that given a graph $G$ on $n$ vertices and $m$ edges, the algorithm provides a ${\rm b}^{\ast}$-coloring in $G$. ${\rm b}^{\ast}$-colorings are useful to obtain bounds for the b-chromatic number of graphs \cite{Z3}. Note that unlike the b-chromatic number, ${\rm b}^{\ast}(G)=\max {\rm b}^{\ast}(H)$, where the maximum is taken over all connected components $H$ of $G$. For a graph $G$, define $G\vee K_1$ as a graph obtained by adding a vertex $v$ to $G$ and joining $v$ to all the vertices of $G$. It was proved in \cite{Z2} that ${\rm b}(G)={\rm b}^{\ast}(G\vee K_1)-1$. A proper coloring $c$ is called z-coloring in \cite{Z1,Z2} if $c$ is a Grundy-coloring using say $k$ colors such that $c$ contains a nice vertex of color $k$. Denote by z$(G)$ the maximum number of colors in a z-coloring of $G$. Define $m^{\ast}$ as the largest integer $k$ such that a vertex of degree at least $k$ has $k$ neighbors in $G$ of degrees at least $k$. It was proved in \cite{Z2} that $z(G)\leq {\rm b}^{\ast}(G) \leq m^{\ast}(G)+1\leq \Delta(G)+1$. A graph $G$ is ${\rm b}^{\ast}$-monotonic if ${\rm b}{\ast}^(H_2)\leq {\rm b}^{\ast}(H_1)$ for every induced subgraph $H_1$ of $G$ and every induced subgraph $H_2$ of $H_1$. Some ${\rm b}^{\ast}$-monotonic and $z$-monotonic families of graphs were obtained in \cite{Z2}. We will use the concept of $z$-atoms from \cite{Z1}. Let $H$ and $G$ be two graphs and $c$ a proper coloring of $H$. We say $(H,c)$ is embedded in $G$ if $H$ is a subgraph of $G$ and for each $u,w\in V(H)$ if $c(u)=c(w)$ then the copies of $u$ and $w$ in $G$ are not adjacent. It was proved in \cite{Z1} that for each integer $k\geq 1$, there exists a family of graphs (called $k$-$z$-atoms) such that for any graph $G$ if $z(G)\geq k$ then there exists a graph $H$ from the family of $k$-$z$-atoms and a canonical $z$-coloring $c$ of $H$ using $k$ colors such that $(H,c)$ is embedded in $G$. The converse of this fact does not hold \cite{Z1}.

\noindent The b-chromatic number of graphs of high girth has been research subject of many papers e.g. \cite{CLS,DFPRZ,MS}. In a graph $G$ with non-increasing degree sequence $d_1, \ldots, d_n$, define $m(G)=\max \{i:d_i\geq i-1\}$. Obviously, $b(G)\leq m(G)$. It was proved in \cite{CLS} that if $G$ is a graph of girth at least 7, then ${\rm b}(G) \geq m(G)-1$. A parameter related to ${\rm b}^{\ast}$-coloring analogous to $m(G)$ is $m^{\ast}$ (\cite{Z2}) which we defined earlier. In Theorem \ref{mainthm} we prove ${\rm b}^{\ast}(G)=m^{\ast}(G)+1$ for graphs of girth at least $7$. Note that this fact does not hold for graphs of girth $5$. Since for Petersen graph $H$, it can be easily observed that ${\rm b}^{\ast}(H)=3$ but $m^{\ast}(H)+1=4$.

\noindent Our proofs in this paper are based on the list-coloring techniques. Let $L$ be an assignment of colors to the vertices of $G$. We say $G$ is $L$-list-colorable (or shortly $L$-colorable) if there exists a proper coloring $c$ of $G$ such that $c(v)\in L(v)$, for each $v\in G$. A strong result of Erd{\rm \H{o}}s, Rubin and Taylor \cite{ERT} asserts that if for each $v\in G$, $|L(v)|=d_G(v)$ then $G$ is $L$-list-colorable unless each block in $G$ is either a complete graph or an odd cycle. This important result was also proved by Borodin in \cite{B} and reported almost recently in \cite{CR}. We add two more results in this regard which will be used in the proofs.

\begin{prop}

\noindent (i) Let $G$ be a connected graph and $L$ a list assignment to the vertices of $G$ such that for each $u\in V(G)$, $|L(u)|=d_G(u)$. If $G$ is not $L$-colorable. then every block in $G$ is isomorphic to a complete graph or an odd cycle.

\noindent (ii) Let $G$ and $L$ be as in $(i)$, if there exists a vertex $u\in G$ such that $|L(u)|>d_G(u)$, then $G$ is $L$-colorable.

\noindent (iii) With the assumptions of $(i)$, if $uw$ is an edge in $G$ such that $u$ is not cut-vertex, then $L(u)\subseteq L(w)$. If $u$ and $w$ are not cut-vertex then $L(u)=L(w)$.
\label{prop}
\end{prop}

\noindent \begin{proof}
Part $(i)$ was proved in \cite{B, ERT}. To prove $(ii)$, arrange the vertices of $G$ in non-increasing form in terms of their distance from $u$. For each vertex $w\not=u$,there are at most $d_G(w)-1$ colors used in the previously scanned neighbors of $w$. Assign an available color from $L(w)$ to vertex $w$. Finally, in order to assign a color to $u$, since $|L(u)|>d_G(u)$ then an available color exists for $u$ in $L(u)$.

\noindent To prove $(iii)$, assume to the contrary that $\alpha \in L(x)\setminus L(y)$. Graph $G-x$ is connected. Define a new list for each $u\in G-x$ as $L'(u)=L(u)$ if $u\in N_G(x)$. And $L'(u)=L(u)\setminus \{\alpha\}$, if $u\in N_G(x)$. For each $u\in G-x$, we have $|L'(u)|\geq d_{G-x}(u)$. For $y$ we have $L'(y)=L(y)$ and $|L'(y)|> d_{G-x}(y)$. If follows by $(ii)$ that $G-x$ is $L'$-colorable. By assigning color $\alpha$ to $x$ we obtain an $L$-coloring of $G$, a contradiction. Then $L(x)\subseteq L(y)$. If in addition, $y$ is not cut-vertex by repeating the same argument for $y$ we obtain $L(y)\subseteq L(x)$. Then $L(x)=L(y)$, as desired.
\end{proof}

\noindent The following lemma concerning the list colorability of cycles will be used in the proofs.

\begin{lemma}
Let $G$ be a graph isomorphic to a cycle. Let $L$ be an assignment of lists such that for each $u$, $|L(u)|=2$. Then $G$ is $L$-colorable unless $G$ is odd cycle and for each $u$ and $w$, $L(u)=L(w)$.\label{lem}
\end{lemma}

\noindent \begin{proof}
Let $V(G)=\{y_1, \ldots, y_k\}$. Let also $L(y_1)\not=L(y_k)$. Assign a color $r\in L(y_1)\setminus L(y_k)$ to $y_1$. Suppose that for $j\geq 1$, colors of $y_1, \ldots, y_j$ have been determined. There exists a color $s$ in $L(y_{j+1})\setminus L(y_j)$. Assign $s$ to $y_{j+1}$. At last, since $r\not\in L(y_k)$ then there exists a color say $t$ in $L(y_k)$ other than $r$ and $L(y_{k-1})$. Assign $t$ to $y_k$ to obtain an $L$-list-coloring of $G$. It follows that if $G$ is not $L$-colorable then all lists are identical. But in case that all lists are identical, $G$ is $L$-colorable if and only if $G$ is even cycle. This completes the proof.
\end{proof}

\section{Results}

\noindent The first result concerns the ${\rm b}^{\ast}$-chromatic number of graphs with girth at least $7$.

\begin{thm}
Let $G$ be a graph of girth at least 7. Then ${\rm b}^{\ast}(G)=m^{\ast}(G)+1$.\label{mainthm}
\end{thm}

\noindent \begin{proof}
It is enough to prove ${\rm b}^{\ast}(G)\geq m^{\ast}(G)+1$. Write $m=m^{\ast}(G)$ and $k={\rm b}^{\ast}(G)$. The inequality holds for $m=0$ or $m=1$. Assume hereafter that $m\geq 2$. There exist $u$ and $m$ neighbors $w_1, \ldots, w_m$ of $u$ such that $d_G(w_i)\geq m$. For each $i\in \{1, \ldots, m\}$, take a subset $A_i\subset N(w_i)- u$ such that $|A_i|=m-1$. Define the following
$$P=\{u\}\bigcup \{w_1, \ldots, w_m\} \bigcup (\bigcup_i A_i).$$
\noindent Since the girth is at least $7$ then $A_i\cap A_j=\varnothing$ and $A_1\cup \cdots \cup A_m$ is an independent set in $G$. It follows that $G[P]$ is isomorphic to an induced tree. Define a coloring $c$ on $G[P]$ using colors from $C=\{1, \ldots, m+1\}$. Define $c(u)=m+1$, $c(w_i)=i$ and for each $1\leq i \leq m$, assign bijectively all colors in $\{1, \ldots, m\}\setminus \{i\}$ to the vertices of $A_i$. Coloring $c$ is proper on $G[P]$ and all colors in $\{1, \ldots, m+1\}\setminus \{i\}$ appear in $N_{P}(w_i)$, for each $i$. It follows that $w_i$ is a b-vertex of color $i$ and then $u$ is a nice vertex of color $m+1$. In the rest of the proof we first show that $c$ can be extended to a proper coloring of whole $G$ using the colors $1, \ldots, m+1$.

\noindent {\bf Claim 1:} For each $x\not\in V(P)$, $|N_G(x)\cap P|\leq 1$.

\noindent {\bf Proof of Claim 1:} Assume to the contrary that a vertex $x\not\in V(P)$ has two neighbors $v_1, v_2$ in $P$. The distance of $v_1$ and $v_2$ in the tree $G[P]$ is at most $4$. Then the path between $v_1$ and $v_2$ in $G[P]$ together with the edges $xv_1$ and $xv_2$ form a cycle of length at most $6$, a contradiction. Hence the claim holds. Now, define $D=\{x\in V(G)\setminus P: d_G(x)\geq m+1\}$.

\noindent {\bf Claim 2:} $\Delta(G[D])\leq m$.

\noindent {\bf Proof of Claim 2:} Otherwise, let $x\in D$ be such that $d_{G[D]}(x)\geq m+1$. Then there exist $x_1, \ldots, x_{m+1}$ from $N_G(x)\cap D$. We have $d_G(x_i)\geq m+1$, for each $1\leq i\leq m+1$. This means that $x$ has $m+1$ neighbors each of them of degree at least $m+1$. This contradicts the definition of $m^{\ast}$. This proves Claim {\bf $2$}.

\noindent Define the following list to each vertex $x\in D$.

\noindent $L(x)=C\setminus \{c(p): p\in N_G(x)\cap P\}$.

\noindent By Claim 1 we have $|L(x)|\geq m$.

\noindent {\bf Claim 3:} $G[D]$ is $L$-colorable.

\noindent {\bf Proof of Claim 3:} Assume to the contrary that $Y\subseteq D$ is non-empty and minimal such that $G[y]$ is not $L$-colorable. Since $Y$ is minimal then $G[Y]$ is connected. Also, for each $y\in Y$, $G[Y\setminus y]$ is $L$-colorable. If for some $y\in Y$, $d_{G[Y]}(y)<|L(y)|$ then using Proposition \ref{prop} $(ii)$ we obtain that $G[Y]$ is $L$-list-coloring, a contradiction.
%after $L$-list-coloring of $G[Y\setminus y]$, at least one color is not appeared in the neighborhood of $y$. By assigning that color to $y$ we obtain an $L$-list-coloring of $G[Y]$, a contradiction.
It follows that for each $y\in Y$, $d_{G[Y]}(y)\geq |L(y)|$. Using the latter fact, Claim 2 and $|L(x)|\geq m$ (which holds for all $x\in D$), we obtain $m\geq d_{G[Y]}(y)\geq |L(y)| \geq m$. Hence,
$$\forall y\in Y, d_{G[Y]}(y)= m.$$
\noindent It follows that $G[Y]$ is $m$-regular connected graph. Also $d_{G[Y]}(y)=m$ implies $|L(y)|=m$ which shows that every vertex $y\in Y$ has exactly one neighbor in $P$. Denote the very unique neighbor of $y$ by $p(y)$. We bifurcate the proof in terms of $m$.

\noindent {\bf Case 1:} $m\geq 3$

\noindent Consider $G[Y]$ with lists satisfying $d_{G[Y]}(y)=|L(y)|$, for each $y$. Apply Proposition \ref{prop} for $G[Y]$ since $G[Y]$ is connected and not list-colorable. It implies that each block in $G[Y]$ should be either a complete graph or an odd cycle. The condition on girth implies that each block is isomorphic to either $K_2$ or an odd cycle of length at least $7$. If $G[Y]$ contains exactly one block then $\Delta(G[Y])\leq 2$, which contradicts $m\geq 3$. Hence, assume that $G[Y]$ contains more than one block and then one of them say $B$ is an end-block. Block $B$ contains exactly one cut-vertex from $G[Y]$. Let $z$ be a non-cut-vertex in $B$. Then all edges incident to $z$ belong to $B$. Now, if $B$ is isomorphic to $K_2$ then $d_B(z)=d_{G[Y]}(z)=1$ and if $B$ is isomorphic to an odd cycle then $d_B(z)=d_{G[Y]}(z)=2$. Both cases contradicts the fact that $d_{G[Y]}(z)=m\geq 3$. This proves Claim 3 when the case 1 holds.

\noindent {\bf Case 2:} $m=2$

\noindent Using the latter argument for the structure of $G[Y]$, the regularity and minimality of $G[Y]$ shows that it is $2$-regular and then isomorphic to a cycle. Let $V(G[Y])=\{y_1, \ldots, y_r\}$. We have $|L(y_i)|=2$ and $L(y_i)\subseteq \{1,2,3\}$. Apply Lemma \ref{lem} for $G[Y]$ and obtain that $G[Y]$ is list-colorable unless $L(y_i)=L(y_j)$. It follows that for each $y_i$ and $y_j$, $p(y_i)=p(y_j)$ (recall that $p(y)$ is the unique neighbor of $y$ in $P$). Since $m=2$ then $G[P]$ consists of an induced path consisting of $a_1, w_1, u, w_2, a_2$, where $c(u)=3$, $c(w_1)=c(a_2)=1$ and $c(w_2)=c(a_1)=2$. Then no two vertices of same color has distance less than $3$. Take an edge $xy$ from $G[Y]$. If $p(x)=p(y)$ then a triangle is formed on $x, y, p(x)$, a contradiction. But if $p(x)\not=p(y)$ then since $p(x)$ and $p(y)$ have identical colors then their distance in $P$ is $3$ and hence they together with the edges $xp(x), xy, yp(y)$ form a cycle of length $6$, contradiction. This argument rules out the non-list-colorability in case $m=2$. Claim $3$ holds for the case $2$ and its proof is completed.

\noindent We conclude that $G[D]$ is list-colorable. It follows that the partial ${\rm b}^{\ast}$-coloring of $P$ is extended to a proper coloring (yet denoted by $c$) of $G[P\cup D]$ using $m+1$ colors.

\noindent To complete the proof we have a final simple step. We have a partial ${\rm b}^{\ast}$-coloring $c$ of $G[P\cup D]$ with $m+1$ colors. We prove that $c$ can be extended to a proper coloring of $R=V(G)\setminus (P\cup D)$ using $m+1$ colors. Consider an arbitrary ordering $\sigma$ of vertices in $R$. Apply a modified greedy coloring of $R$ with respect to $\sigma$. At each step of the coloring procedure, to color a vertex say $w$ of $R$, assign a smallest number which is not appeared in $N_{P\cup D}(w)$ and also in those vertex in $N_R(w)$ which have been previously colored. The procedure finishes the coloring of $R$ using $1, \ldots, m+1$. Otherwise, let it finishes the coloring process at a maximal subgraph $S$ of $R$. Let $z\in R\setminus S$. Since $d_G(z)\leq m$ then there is an admissible color available for $z$. By assigning this color to $z$ we obtain a proper coloring of $S\cup \{z\}$ using $m+1$ colors. This contradiction shows that $c$ can be extended to entire $R$, as desired.
\end{proof}

\noindent The following result is easily obtained.

\begin{cor}
Let $G$ be a graph of girth at least 7. Then $G$ is ${\rm b}^{\ast}$-monotonic.
\end{cor}

\noindent \begin{proof}
Let $H_1$ be an induced subgraph of $G$ and $H_2$ be an induced subgraph of $H_1$. Both $H_1$ and $H_2$ have girth at least $7$. It follows by Theorem \ref{mainthm} that ${\rm b}^{\ast}(H_1)=m^{\ast}(H_1)+1$ and ${\rm b}^{\ast}(H_2)=m^{\ast}(H_2)+1$. To prove the corollary we prove $m^{\ast}(H_2)\leq m^{\ast}(H_1)$. Let $m^{\ast}(H_2)=r$. It follows that $H_2$ (then $H_1$) contains a vertex $x$ such that $x$ has $r$ neighbors in $H_2$ (then in $H_1$) such that their degrees in $H_2$ (then in $H_1$) is at least $r$. If follows that $m^{\ast}(H_1)\geq r=m^{\ast}(H_2)$. \end{proof}

\noindent The following result involves graphs of girth at least $6$ with $m^{\ast}=3$.

\begin{prop}
Let $G$ be a graph of girth $6$ such that $m^{\ast}(G)=3$. Then ${\rm b}^{\ast}(G)=4$
\end{prop}

\noindent \begin{proof}
Let $P=\{u\}\cup \{w_1, w_2, w_3\} \cup A_1 \cup A_2 \cup A_3$ be as in the proof of Theorem \ref{mainthm}. The condition on girth implies that $G[P]$ is an induced tree. Define a coloring $c$ on $P$ as $c(u)=4$, $c(w_i)=i$ and such that all colors in $\{1,2,3\}\setminus \{i\}$ appear in $A_i$. In this coloring $u$ is nice vertex of color $4$. For each vertex $v\not\in P$, let $L$ be a list obtained by removing the colors appeared in $N_G(v)\cap P$ from $C=\{1, \ldots, 4\}$. If $G\setminus P$ is $L$-colorable then $c$ can be extended to a proper $4$-coloring of the whole graph. Otherwise, assume to the contrary that $G\setminus P$ is not $L$-colorable. We are going to obtain a contradiction.

\noindent Take a minimal $Y\subseteq V(G)\setminus P$ such that $G[Y]$ is not $L$-colorable. As in the proof of Theorem \ref{mainthm}, $G[Y]$ is connected and $d_{G[Y]}(v)\geq |L(v)|$, for each $v\in Y$. For each $v$, let $s(v)=|N_G(v)\cap P|$ and $r(v)$ be number of distinct colors appeared in $N_G(v)\cap P$. We have $|L(v)|=4-r(v)$ and $s(v)\geq r(v)$. We obtain the following for each $v\in Y$
$$d_G(v)\geq d_{G[Y]}(v)+s(v)\geq |L(v)|+s(v)=4-r(v)+s(v)\geq 4 \Rightarrow d_G(v)\geq 4~~~~~{\bf (1)}$$
\noindent {\bf Claim 1:} $\Delta(G[Y])\leq 3$.

\noindent To prove Claim 1, let $x\in Y$ be a vertex with $4$ neighbors of degree at least $4$ in $Y$. This contradicts $m^{\ast}(G)=3$. Since $|L(v)|\leq \Delta(G[Y])\leq 3$, then each $v\in Y$ has at least one neighbor in $P$.

\noindent {\bf Claim 2:} For each $v\in Y$, $v$ is not adjacent to $u$.

\noindent To prove the claim, let $v\in Y$ and $uv\in E(G)$. Since the girth is at least $6$ then $N_G(v)\cap P=\{u\}$. Then $|L(v)|=3$. We have $d_{G[Y]}(v)\geq 3$. Let $y\in Y$ be adjacent to $v$. Vertex $y$ has a neighbor in $P$. We obtain a cycle of length at most $5$. This contradiction implies that for each $v\in Y$, $N_G(v)\cap \{u\}=\varnothing$.

\noindent For each $v\in Y$, define $\sigma(v)=\{i\in \{1,2,3\}:~N_G(v)\cap B_i \not=\varnothing\}$, where $B_i=A_i\cup \{w_i\}$.

\noindent {\bf Claim 3:} For each edge $xy\in G[Y]$, $\sigma(x)\cap \sigma(y)=\varnothing$.

\noindent To prove Claim 3, let $i\in \sigma(x)\cap \sigma(y)$. Pick $q_x\in N_G(x)\cap B_i$ and $q_y\in N_G(y)\cap B_i$. If $q_x=q_y$ then $x,y,q_x$ form a triangle. If $q_x\not=q_y$ then their distance in $B_i$ is at most $2$. The very path with the edges $xy,xq_x,yq_y$ form a cycle of length $5$. This contradiction proves Claim 3.

\noindent {\bf Claim 4:} For each $v\in G[Y]$, $\sigma(v)\leq 2$.

\noindent To prove Claim 4, assume to the contrary that $\sigma(v)=\{1,2,3\}$. Vertex $v$ is not adjacent to $u$ then $4\in L(v)$ and $|L(v|\geq 1$. By $d_{G[Y]}(v)\geq |L(v)|$, let $w\in Y$ be adjacent to $v$. By Claim 3, $\sigma(w)\cap \{1,2,3\}=\varnothing$, which is impossible since $\sigma(w)\subseteq \{1,2,3\}$. This proves Claim 4.

\noindent For each $i\in \{1,2,3\}$, define $X_i=\{x\in Y: \sigma(x)=\{i\}\}$ and $Z_i=\{x\in Y: \sigma(x)=\{1,2,3\}\setminus \{i\}\}$. Let also $X=X_1\cup X_2\cup X_3$ and $Z=Z_1\cup Z_2\cup Z_3$. We have $V(Y)=X\cup Z$ and $Z$ is independent because any two subsets of $\{1,2,3\}$ of cardinality $2$ intersects. We have also $4\in L(z)$, for each $z\in Z$. The reason is that $z$ is not adjacent to $u$ and all its neighbors in $P$ have colors in $\{1,2,3\}$. Assign color $4$ to all vertices in $Z$. Since $Z$ is independent the assignment is proper. In the following, to get a contradiction we prove that $G[Y][X]$ is $L$-colorable. This will imply that $G[Y]$ is $L$-colorable.

\noindent For each $x\in X_i$, we have $\sigma(x)=i$. Hence, $x$ has a unique neighbor in $P$. Denote the very neighbor by $p(x)$. It follows that $N_G(x)\cap P=\{p(x)\}$. Also $L(x)=\{1, \ldots, 4\}\setminus \{c(p(x))\}$. Then $|L(x)|=3$. It implies that $d_{G[Y]}(x)=3$. For each $x\in X$, define $q(x)=|N_{G[Y]}(x)\cap Z|$. We obtain from $d_{G[Y]}(x)=3$
that
$$d_{G[Y][X]}(x)=3-q(x)~~~~{\bf (2)}$$
\noindent Define a new lists for the elements of $X$ as $L'(x)=L(x)$ if $q(x)=0$ and $L'(x)=L(x)\setminus \{4\}$ if $q(x)\geq 1$. It follows that  $|L'(x)|=3$ if $q(x)=0$ and $|L'(x)|=2$ if $q(x)\geq 1$. By {\bf (2)} we have the following for each $x\in X$
$$|L'(x)|\geq d_{(G[Y])[X]}(x)~~~~{\bf (3)}$$
\noindent Now, let $Q$ be a connected component of $(G[Y])[X]$ which is not $L'$-list-colorable. Then ${\bf (3)}$ holds for each vertex of $Q$. If for some $x\in Q$, $|L'(x)|> d_{Q}(x)$ then by Proposition \ref{prop} $(iii)$, $Q$ should be $L'$-list-colorable. This contradiction shows that
$$\forall v\in Q, |L'(v)|= d_{Q}(v)~~~~{\bf (4)}$$
\noindent From definition of $L'$ and $d_{(G[Y])[X]}(x)=3-q(x)$ we obtain the following. (1) If $q(x)=0$ then $d_Q(x)=3$. (2) If $q(x)\geq 1$ then $2=d_Q(x)=3-q(x)$ and then $q(x)=1$. It follows that for each $v\in Q$, $d_Q(v)\in \{2,3\}$. We obtain
$$\delta(Q)\geq 2~~~~~~~{\bf (5)}$$
\noindent Applying Proposition \ref{prop} $(i)$ for connected graph $Q$, we conclude that each block in $Q$ is either complete subgraph or odd cycle. Since the girth of $Q$ is at least $6$ then each block in $Q$ is either $K_2$ or an odd cycle of length at least $7$. If there is exactly one block in $Q$ then $\delta(Q)$ implies that it is a cycle. Otherwise, let $B$ be an end-block in $Q$. Clearly $B$ should be a cycle block. It follows that $Q$ contains a block isomorphic to an odd cycle of length at least $7$. Let $x_1, \ldots, x_4$ be four consecutive vertices in a block of $Q$ which are not cut-vertex. Each $x_j$ has degree 2 in this block. We have $|L'(x_j)|=d_Q(x_j)=3-q(x_j)=2$. Then $q(x_j)=1$, for each $j\in \{1,2,3,4\}$. Apply Proposition \ref{prop} $(iii)$ for $Q$ and the non-cut-vertex vertices $x_1, \ldots, x_4$, we obtain that
$$\forall i,j\in \{1, \ldots, 4\}, L'(x_i)=L'(x_j)~~~~~~~{\bf (6)}$$
\noindent For each $j$, $x_j\in X$ and $q(x_j)=1$ then $L'(x_j)=\{1,2,3\}\setminus \{c(p(x_j))\}$. Recall that $p(x_j)$ is the unique neighbor of $x_j$ in $P$. These facts together with (6) imply
$$c(p(x_1))=c(p(x_2))=c(p(x_3))=c(p(x_4))~~~~~~~{\bf (7)}$$
\noindent By the pigeonhole principle, there exist $x_r$ and $x_s$ and an index $i$ such that $x_r, x_s\in X_i$. By (7) we have $c(p(x_r))=c(p(x_s))$. On the other hand each color $1,2,3$ appears exactly once in $B_i$. Then $p(x_r)=p(x_s)$. The distance between $x_r$ and $x_s$ is at most $3$. Vertices $x_r$ and $x_s$ are both adjacent to $p(x_r)$. We obtain a cycle of length at most $5$. This contradiction shows that every connected component of $(G[Y])[X]$ is $L'$-list-colorable. It follows that $(G[Y])[X]$ is $L'$-list-colorable. Combination of the coloring of $G[Y][X]$ with a coloring which assigns color $4$ to all vertices of $Z$ we obtain an $L$-list-coloring of entire $G[Y]$. It implies that the ${\rm b}^{\ast}$-coloring of $G[P]$ can be extended to a coloring of $G\setminus V(P)$ using $4$ colors in which $u$ remains a nice vertex of color $4$. This completes the proof.
\end{proof}

\noindent Theorem \ref{z} concerns the $z$-chromatic number of graphs of girth at least $2k+2$, where $k\geq 3$ is fixed integer. As proved in \cite{Z1}, there exists exactly one tree having smallest number of vertices and with $z$-chromatic number $k$. According to \cite{Z1}, this tree (denoted by $R_k$) is a rooted tree on $(k-3)2^{k-1}+k+2$ vertices and admits a canonical $z$-coloring $c$ using $k$ colors such that the color of its root (say $u$) is $k$. It is also known from the structure of $R_k$ that its radius is $k$. The following fact about $R_k$ is used in the following proof. Draw $R_k$ top-down such that $u$ is the most top vertex and its b-neighbors $w_1, \ldots, w_{k-1}$ of colors $1, \ldots, k-1$ are in the first level $L_1$. In the second level $L_2$, there are necessary neighbors of $w_1, \ldots, w_{k-1}$ which make them b-vertex. Let $L_3, \ldots$ be the next and lower levels in $R_k$. For each $i\geq 2$ and each $w\in L_i$ of color say $j$, there are exactly $j-1$ children of $w$ in $L_{i+1}$ and their colors are $1, \ldots j-1$. It follows that the lowest level of $R_k$ consists only of vertices having color $1$.

\noindent Note that in contrast to ${\rm b}^{\ast}$-coloring, no girth condition guarantees that $z(G)=m^{\ast}(G)+1$. To prove this fact consider a rooted tree obtained by a vertex say $u$ (as the root) of degree $k$ such that each neighbor of $u$ has $k-1$ neighbors of degree one in a lower level. Now, connect two leaves from different branches by a path of an arbitrary length. The resulting graph has arbitrary large girth but with $z$-number $3$. We conclude that $R_k$ (or other $z$-atoms \cite{Z1}) and not $m^{\ast}$ is an essential factor for controlling the $z$-chromatic number of graphs.

\begin{thm}

\noindent (i) Let \(k \geq 3\) and \(G\) be a graph of girth at least $2k+2$, where $k=m^{\ast}(G)+1$. If \(G\) contains \(R_k\) as ordinary subgraph, then \(z(G)=k\).

\noindent (ii) Let \(k \geq 3\) and \(G\) be a graph on at least $(k-3)2^{k-1}+k+2$ vertices such that $m^{\ast}(G)=\delta(G)=k-1$ and girth of $G$ is at least $2k+2$. Then $z(G)=k$.\label{z}
\end{thm}

\noindent\begin{proof}
If $k=3$ then existence of $R_3$ is equivalent to $m^{\ast}=2$. In this case girth of $G$ is at least 8 then Theorem \ref{mainthm} is applied. The ${\rm b}^{\ast}$-coloring of the set $P$ with $3$ colors in the proof of Theorem \ref{mainthm} is a $z$-coloring and the coloring can be extended to a Grundy-coloring of whole $G$. It follows that $z(G)=m^{\ast}+1=3$. Assume hereafter that $k\geq 4$.

\noindent Let \(T\) denote the subgraph of $G$ isomorphic to \(R_k\) with \(u\) as its root. Let \(c: V(T) \rightarrow \{1, \ldots , k\}\) be the canonical \(z-\)coloring on \(T\). Since the girth is at least $2k+2$ then \(T\) is an induced subgraph of \(G\). We prove that every vertex in \(G\setminus T\) sees at most one color on its neighbors in \(V(T)\).
As we explained in the previous paragraph concerning the structure of $R_k$, the vertices in the last level of $R_k$ have color $1$ in its canonical coloring $c$. It means that if a vertex $x\in V(G)\setminus V(T)$ is adjacent to two vertices of $T$ with distinct colors then a cycle of length less than $2k+2$ is formed. It follows that for each \(x \in V(G) \setminus V(T)\)
		\[
		|c(N_G(x)\cap V(T))| \leq 1.
		\tag{1} \label{1}
		\]
\noindent Now define the set \(D = \{x \in V(G) \setminus V(T) : d_G(x)\geq k \}\). For each \(x \in D\), we define the list of available colors as follow.
		\[
		L(x) = \{1, \ldots k\} \setminus c(N_G(x) \cap V(T)).
		\]
		By \eqref{1}, for each \(x \in D\):
		\[
		|L(x)| \geq k-1.
		\tag{2} \label{2}
		\]
For the induced subgraph \(G[D]\) we have
		 \[
		 \Delta (G[D]) \leq k-1
		 \tag{3} \label{3}
		 \]
The reason is that if a vertex \(x \in D\) has \(k\) neighbors in \(D\) then all of them have degree at least \(k\) in \(G\) so \(m^*(G) \geq k\), a contradiction. Then \eqref{3} is valid. Now consider a connected component \(Q\) of \(G[D]\).

\noindent Case 1: \( \Delta(Q) \leq k-2\)

\noindent In this case by \eqref{2}, we have \(|L(x)| \geq k-1 \geq \Delta(Q)+1 \), for each \(x \in V(Q)\). Then \(Q\) is \(L\)-colorable using a greedy procedure.

\noindent Case 2: \( \Delta(Q)= k-1 \)

\noindent \noindent In this case, we have \(k \geq 4\) then \( \Delta(Q) \geq 3\). Also the girth of $Q$ is at least $2k+2$. Then no block in \(Q\) is complete graph on at least $3$ vertices. Since\( \Delta(Q) \geq 3\) then $Q$ is not isomorphic to cycle graph. Let $B$ be a pendant block in $Q$. If $B$ is isomorphic to either $K_2$ or a cycle then $B$ contains a vertex of degree at most $2$ and its degrees in $B$ and $Q$ are equal. It implies that bo block in $Q$ can be cycle graph or $K_2$. By Proposition \ref{prop} $(i)$, \(Q\) with lists of size at least \(k-1\) for its vertices is $L$-colorable.

\noindent By applying this argument for all connected components of \(G[D]\), we obtain that coloring \(c\) from \(T\) can be extended to a proper coloring for \(G[V(T) \cup D]\) using colors \(\{ 1, \ldots k\}\). Now set \(R = V(G) \setminus (V(T) \cup D)\). By the definition of \(D\), for each vertex \(x \in R\) we have
		  \[
		  d_G(x) \leq k-1
		  \tag{4} \label{4}
		  \]
Let \(S \subseteq R\) be the maximal set such that \(c\) can be extended to \(G[V(T) \cup D \cup S]\). If \(S \neq R\), choose a vertex \(x \in R \setminus S \). All the colored neighbors of \(x\) are in \(V(T) \cup D \cup S\) and by  \eqref{4} they are at most \(k-1\) vertices. So the maximum number of colors that are used in the neighborhood of \(x\) is \(k-1\) and thus there exists at least one color say \( \alpha\) available for \(x\). By \(c(x)= \alpha\) the coloring will be extended to \(S \cup \{x\}\), a contradiction with the maximality of \(S\). Hence, \(S=R\). So the coloring of \(T\) is extended to a proper coloring $c^+$ of the entire \(G\) using \(k\) colors.

\noindent Coloring $c^+$ is not necessarily Grundy-coloring on $V(G)\setminus V(T)$. We revise $c^+$ in order to make it Grundy-coloring. Starting from the vertices of color $2$, if a vertex has not any neighbor of color $1$ then switch its color to $1$. Then scan all vertices $u$ of color $3$. If colors $1,2$ appear in $N_G(u)$ then its color remains unchanged, otherwise switch its color to the smallest one which is not appeared in $N_G(u)$. Then we scan all vertices of color $4$ and make a similar revision. Repeat this scenario for all other colors. We eventually obtain a Grundy-coloring in which $u$ is nice and its neighbors $\{u_1, \ldots, u_{k-1}\}$ are b-vertices of colors $\{1, \ldots k-1\}$, respectively. It follows that $z(G)\geq k$. Since $m^{\ast}(G)=k-1$ then $z(G)=k$.

\noindent To prove $(ii)$, note that $G$ contains any tree $T$ as ordinary subgraph if $\Delta(T)\leq k-1$, the radius of $T$ is at most $k$ and $|V(G)|\geq |V(T)|$. This fact is easily proved by induction on $|V(T)|$. By the assumption $|V(G)|\geq |V(R_k)|$. Then $R_k$ satisfies these conditions. It follows that $R_k$ is isomorphic to a subgraph of $G$. By $(i)$, $z(G)=k$.
	\end{proof}

\noindent We finish the paper by generalizing Theorem \ref{z} to arbitrary $z$-atoms. Recall from the introduction that if  $G$ is a graph with $z(G)\geq k$ then there exists a $k$-$z$-atom $H$ such that $(H,c)$ is embedded in $G$, where $c$ is the canonical $z$-coloring of $H$ using $k$ colors. The converse of this fact does not necessarily hold.

\begin{prop}
Let $k\geq 4$, $H$ be a $k$-$z$-atom and $c$ be the canonical coloring of $H$. Let $\ell$ be the maximum distance between two vertices of $H$ whose colors are distinct in $c$. Let $G$ be a graph which contains $H$ as subgraph such that every cycle in $G$ have length at least $\ell+3$ unless the cycle is entirely contained in $H$. Let $m^{\ast}(G)=k-1$. Then $z(G)=k$.
\end{prop}

\noindent\begin{proof}
The proof is similar to the proof of Theorem \ref{z}. Note that $H$ is an induced subgraph of $G$. Otherwise, let $u$ and $w$ be two vertices in $H$ but $uw\in E(G)$. Then $c(u)\not= c(w)$ and there exists a path of length at most $\ell$ in $H$. We obtain a cycle of length at most $\ell+1$ which is not entirely in $H$. Note also that for each $x \in V(G) \setminus V(T)$, $|c(N_G(x)\cap V(T))| \leq 1$. Now define \(D = \{x \in V(G) \setminus V(T) : d_G(x)\geq k \}\). For each \(x \in D\), define $L(x)= \{1, \ldots k\} \setminus c(N_G(x) \cap V(T))$. By the previous fact, for each $x \in D$, $|L(x)| \geq k-1$. Also for the induced subgraph \(G[D]\) we have $\Delta (G[D]) \leq k-1$. Now an argument similar the one in the proof of Theorem \ref{z} shows that $G[D]$ is $L$-colorable. It means that $c$ can be extended to a proper coloring $c$ of $G[D]$ using $1, \ldots, k$. As in the proof of Theorem \ref{z}, $c$ can be extended to a Grundy-coloring of $G$. This completes the proof.
\end{proof}

\end{document}